\documentclass[10pt]{article}
\usepackage[utf8]{inputenc}
\usepackage{authblk}

\usepackage[top=25truemm,bottom=25truemm,left=30truemm,right=30truemm]{geometry}

\usepackage{amsmath,amssymb,amsthm,mathtools}
\usepackage[capitalize,noabbrev]{cleveref}
\usepackage{mysty,mythm}

\usepackage{lineno}

\makeatletter
	
	\@addtoreset{equation}{section}
\makeatother

\theoremstyle{definition}
\newtheorem{observation}[theorem]{Observation}

\usepackage{framed} 
\usepackage{txfonts} 
\usepackage{xfrac}
\usepackage{enumitem}
\usepackage{subfigure}
\usepackage{bm}
\usepackage{natbib}
\bibpunct[:]{(}{)}{,}{a}{}{,} 
\usepackage{algorithm}
\usepackage{algpseudocode}
\usepackage{empheq}
\usepackage{multicol}
\usepackage{overpic}
\usepackage{extarrows}
\usepackage{here}
\usepackage{comment}
\allowdisplaybreaks

\usepackage{physics}

\newcommand{\transpose}{^\top}

\newcommand{\mycomment}[1]{\textcolor[rgb]{0.862,0.078,0.117}{#1}}

\title{A fluid-particle decomposition approach to\\ matching market design for crowdsourced delivery systems}
\author[1]{Takashi Akamatsu\thanks{Corresponding Author: akamatsu@plan.civil.tohoku.ac.jp}}
\affil[1]{Graduate School of Information Sciences, Tohoku University, Miyagi, Japan}

\author[2]{Yuki Oyama\thanks{Corresponding Author: oyama@shibaura-it.ac.jp}}
\affil[2]{Department of Civil Engineering, Shibaura Institute of Technology, Tokyo, Japan}

\date{\today}

\begin{document}
\maketitle

\begin{abstract}
This paper considers a crowdsourced delivery (CSD) system that effectively utilizes the existing trips to fulfill parcel delivery as a matching problem between CSD drivers and delivery tasks. This matching problem has two major challenges. First, it is a large-scale combinatorial optimization problem that is hard to solve in a reasonable computational time. Second, the evaluation of the objective function for socially optimal matching contains the utility of drivers for performing the tasks, which is generally unobservable private information. 
To address these challenges, this paper proposes a hierarchical distribution mechanism of CSD tasks that decomposes the matching problem into the task partition (master problem) and individual task-driver matching within smaller groups of drivers (sub-problems). We incorporate an auction mechanism with truth-telling and efficiency into the sub-problems so that the drivers' perceived utilities are revealed through their bids. Furthermore, we formulate the master problem as a fluid model based on continuously approximated decision variables. By exploiting the random utility framework, we analytically represent the objective function of the problem using continuous variables, without explicitly knowing the drivers' utilities.
The numerical experiment shows that the proposed approach solved large-scale matching problems at least 100 times faster than a naive LP solver and approximated the original objective value with errors of less than 1\%.


\end{abstract}

\newpage
\tableofcontents 
\newpage

\section{Introduction}
\subsection{Background and motivation}
E-commerce has grown rapidly over the past two decades, as well as during the COVID-19 pandemic, and rising customer expectations of fast, low-price, and punctual delivery have significantly increased the demand for last-mile parcel delivery. 
Last-mile delivery is a challenge for traditional couriers with a limited number of drivers due to the large volume of small but frequent delivery requests, necessitating the development of novel and cost-effective delivery methods. One of the promising approaches aided by the improvements in mobile Internet technology is crowdsourced delivery (CSD), also referred to as crowd-shipping, where delivery tasks are performed by ordinary drivers (e.g., commuters or travelers) who travel anyway for their own purposes, in cooperation with dedicated drivers. 
By distributing tasks to a large pool of potential drivers, CSD enables faster and more cost-effective deliveries than traditional urban logistics. Moreover, because CSD fulfills delivery tasks by utilizing the excess capacity of private vehicles that may not fully used, it can also reduce the negative environmental impact associated with the use of dedicated delivery vehicles, such as traffic, pollutant emissions, and energy consumption \citep{mladenow2016crowd,paloheimo2016transport,simoni2020potential}. 

In urban areas there exist a vast number of trips with different origin-destination (OD) pairs and different lengths in urban areas. We aim at effectively utilizing these existing trips to fulfill last-mile delivery, including requests with different pickup and delivery points with different distances. Mathematically, this can be viewed as a matching problem between CSD drivers and delivery tasks to minimize the additional travel cost of drivers according to the detour for pickup and delivery from the originally planned trips as well as the system operation cost  (\Cref{fig:system}). Although this type of matching problem for CSD has been intensively studied \citep[e.g.,][]{wang2016towards, soto2017matching}, two major challenges remain unaddressed in the literature.

Firstly, the numbers of potential CSD drivers and delivery tasks are generally so large that one cannot solve the matching problem at once with a reasonable computational time. Therefore, partition or reduction of the problem into smaller, more tractable matching problems is necessary. 
Previous studies consolidated pickup points to a limited number of depots \citep[e.g.][]{wang2016towards, archetti2016vehicle, arslan2019crowdsourced} or clustered tasks based on their geographical locations \citep{huang2021solving, elsokkary2023crowdsourced, simoni2023crowdsourced}, allocating the partitioned tasks to driver groups and then solving the matching problem within each group.
These existing task partitions, however, are based solely on the characteristics of delivery tasks and do not explicitly consider the efficiency of the matching problems after the partition. Because the task partition pattern significantly impacts the efficiency of matching, i.e., the resultant travel and operation costs, the task partition has to be done so that the global matching efficiency for all groups of drivers is maximized. This is not likely to be achieved by an arbitrary partition that relies on pickup or delivery locations of requests. Such a partition may generate a set of delivery tasks with a wide range of distances, particularly in cases where pickup and delivery locations of requests are widely distributed over the entire network, leading to an inefficient matching pattern. 


Secondly, because CSD drivers are not dedicated to delivery and travel anyway for their own purposes, their preferences for delivery tasks, such as cost perceptions and willingness to work, can be heterogeneous. In other words, the preferences of drivers for performing a delivery task by taking detours is private information that cannot be generally observed. Previous studies have treated drivers' costs deterministically so that the compensations paid to them can be determined a priori or solely according to the actual distances traveled for delivery \citep[e.g.][]{wang2016towards, archetti2016vehicle, gdowska2018stochastic, arslan2019crowdsourced}. However, the assignment and pricing of delivery tasks can be biased if the heterogeneity of drivers' preferences is not taken into account. 

\begin{figure}[t]
    \centering
    \includegraphics[width=0.75\textwidth]{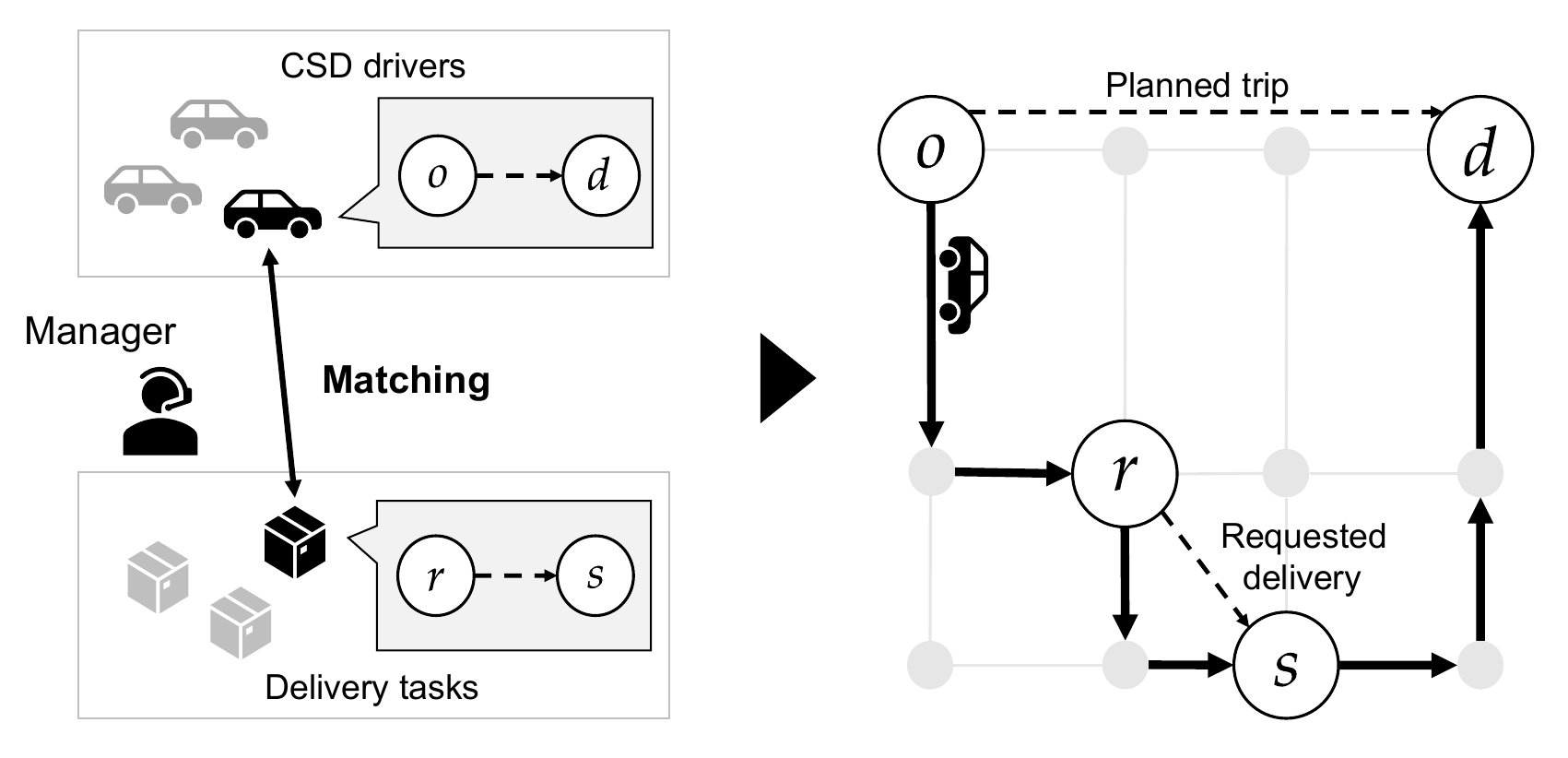}
    \caption{Conceptual diagram of the CSD system.}
    \label{fig:system}
\end{figure}

\subsection{The approach}
In this study, to address these two challenges, we propose a hierarchical distribution mechanism of CSD tasks. 
We consider a transportation network and CSD drivers who plan their travels between different OD pairs. The proposed mechanism involves the following steps: (1) the CSD system manager receives from customers delivery requests whose pickup and delivery points can be any nodes of the network and registers them in the system (e.g., using a mobile application); (2) the manager partitions the registered tasks into nonoverlapping collections and allocates each of them to a specific group of drivers (typically grouped by OD pair); (3) the manager opens auctions for each group of drivers to individually match each task to a driver, where the assignment and prices of tasks are determined based on bids submitted by drivers; (4) the winner of each auction performs the delivery task on the way to his/her destination, and if there is a task for which no one did not submit a bid or whose lowest bid price exceeds the operational cost for a dedicated vehicle operated by the manager, then the task is carried out by the dedicated vehicle.

As described in steps 2 and 3, the proposed mechanism decomposes the large-scale CSD matching problem into two main problems (\Cref{fig:approach}): the task partition (\textit{master problem}) and individual task-driver matching within smaller groups of drivers (\textit{sub-problem}s). 

The task partition problem is the core of the proposed mechanism, which determines the combination of numbers of delivery tasks differentiated by pickup and delivery locations to allocate to each group of drivers.
While an arbitrary partition based on the characteristics of delivery tasks is independent of the sub-problems and may be inefficient \citep{huang2021solving, simoni2023crowdsourced}, this study solves the task partition so that global matching efficiency is achieved approximately but with high accuracy, by proposing a \textbf{fluid-particle hybrid decomposition} approach. 
This approach formulates the task partition (master) problem as its
\textit{fluid} version, where the decision variables are considered to be the number of delivery tasks to allocate each group of drivers and are continuously approximated. It also uses a random utility model (RUM) framework to represent the heterogeneity in drivers' costs for performing delivery tasks. Given that the cost distribution has been estimated using the information revealed during the past auctions, the optimal value functions of the sub-problems can be analytically derived for each task partition pattern, and thus the objective function of the master problem can be represented using the continuous decision variables without solving the sub-problems.


Once delivery tasks are partitioned and allocated to specific groups of drivers via solving the master problem, then the individual matching of tasks with drivers in each group is performed in \textit{particle} units. In this study, we employ an auction mechanism for this matching, wherein the private utilities of drivers can be directly observed through their bids. The price for each delivery task is decided on the basis of the Vickrey-Clarke-Groves (VCG) market mechanism \citep{vickrey1961counterspeculation, clarke1971multipart, groves1973incentives}. 
As such, the auction mechanism simultaneously achieves task assignment at the particle level, pricing, and observation of drivers' preferences. The revealed private information of drivers' utilities for performing tasks allows us to estimate the parameters of their distribution and improve the overall efficiency of matching, by repeating it on a day-to-day basis.

The proposed approach also significantly reduces the computing time for matching. We show that the logit version of the fluidly-approximated master problem coincides with an entropy-regularized optimal transport (EROT) problem. Because it is known that an EROT problem can be solved very efficiently by the Sinkhorn-Knopp algorithm, our approach can quickly find the optimal task partition pattern, i.e., the number of delivery tasks that should be allocated to a group of drivers traveling between each OD pair. The task-driver matching problem is decomposed into smaller sub-problems, which are independent of each other and can be solved in parallel.

\begin{figure}[t]
    \centering
    \includegraphics[width=0.95\textwidth]{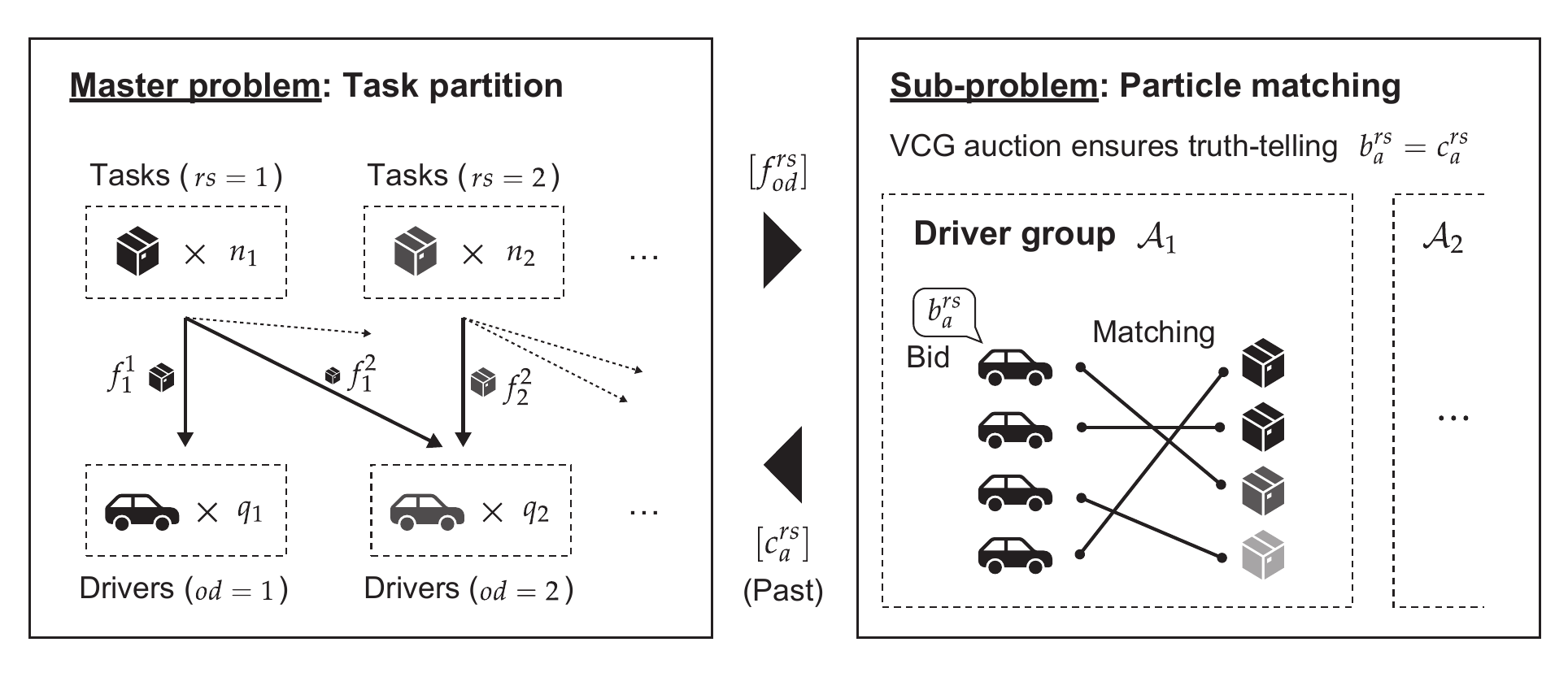}
    \caption{Hierarchical matching mechanism.}
    \label{fig:approach}
\end{figure}

\subsection{Contributions and structure of the paper}
In summary, this study proposes a matching mechanism for large-scale CSD based on the fluid-particle hybrid decomposition approach. To the best of our knowledge, this is the first study on mechanism design for matching markets of CSD systems. Although a few papers studied bid-based matching for a CSD system \citep{kafle2017design, allahviranloo2019dynamic}, they do not discuss market mechanism design or market equilibrium conditions. 
The methodological contributions of this study are summarized as follows:
\begin{enumerate}
    \item We formulate the CSD problem as an optimal matching problem to maximize social surplus, associated with the operation cost of the system manager and the cost of drivers for performing delivery tasks, and show that it is mathematically equivalent to the market equilibrium conditions.
    \item We propose a fluid-particle decomposition approach to address the two major challenges of the matching problem, namely tackling the computational intractability of the combinatorial optimization problem and capturing drivers' unobservable heterogeneous preferences. With this approach, the matching problem is hierarchically decomposed into task partition at the group level (master problem) and task-driver matching at the individual level within each group (sub-problems).
    \item We incorporate an auction mechanism into the individual-level matching problem. During the auctions, drivers' private utilities for performing delivery tasks are revealed through their bids. Moreover, by applying the VCG mechanism that satisfies truth-telling and efficient properties, socially optimal matching and pricing patterns are achieved in particle units.
    \item We formulate the group-level task partition problem as a fluid model to approximately but accurately achieve a globally efficient task partition. By assuming that the distribution of drivers' private disutilities following a RUM framework is estimated from historical sources, the optimal value function of each sub-problem can be analytically derived, thereby allowing us to represent the objective function of the master problem using fluid variables, without solving the individual matching. 
    \item We efficiently solve the matching problem. We show that the fluidly-approximated task partition problem can be reduced to an EROT problem that is efficiently solved. In addition, the decomposed sub-problems are independent of each other and can be solved in parallel. The numerical experiments demonstrate that our approach significantly reduces computing time, at least a hundred times faster than a linear programming (LP) solver, with practically acceptable approximation errors.
\end{enumerate}

The remainder of this paper is structured as follows. \Cref{sec:review} provides the literature review. \Cref{sec:problem} defines and formulates the CSD matching problem and discusses its equivalence to the market equilibrium conditions. \Cref{sec:design} proposes a mechanism design based on the fluid-particle decomposition approach. \Cref{sec:algorithm} presents an efficient solution algorithm for task partition, and \Cref{sec:experiment} shows the numerical experiment. Finally, \Cref{sec:conclusion} concludes the study.

\section{Literature review}\label{sec:review}
The concept of CSD has attracted growing attention both in industry and academia. \cite{le2019supply} and \cite{alnaggar2021crowdsourced} provide recent comprehensive literature reviews, summarizing the state-of-the-art of academic literature as well as various CSD platforms emerging in the world. CSD systems can be categorized by different aspects, such as courier types or pricing/compensation schemes. 
Drivers in a CSD system can be dedicated drivers to delivery tasks but are often considered to be ordinary drivers (e.g., commuters or travelers) who already planned their own trips and travel anyway \citep{le2019supply}. Utilizing already existing trips of ordinary drivers have a significant impact from both operational and social points of view, reducing operational costs and negative environmental impacts \citep{simoni2020potential, voigt2022crowdsourced}. The matching of such drivers with delivery tasks is also called \textit{en-route matching} \citep{alnaggar2021crowdsourced}, as they fulfill delivery tasks en route on the way to their destination. 

The present study focuses on this type of matching problem in urban areas, where a large number of drivers exist and parcel delivery is highly demanded, and proposes a matching mechanism to solve the problem efficiently and to address heterogeneous preferences of crowdsource drivers. The following literature review therefore focuses on the development of (1) efficient matching algorithms and (2) pricing and compensation schemes.

\subsection{Matching algorithms}
An integrated framework of vehicle routing problems (VRPs) with the CSD concept has been studied in the literature \citep[e.g.,][]{archetti2016vehicle, arslan2019crowdsourced, dayarian2020crowdshipping, voigt2022crowdsourced}. Those studies consider a problem in which a company with a limited number of dedicated vehicles outsources some delivery tasks to so-called occasional drivers, often assumed to be in-store customers, who pick up parcels at the store and deliver them on their way back home. VRPs with occasional drivers are often formulated as mixed integer programming. Although various solution approaches have been proposed, only small-sized cases were studied due to the complex nature of the routing problem.

Several CSD studies present techniques to decompose a large-scale matching problem into smaller sub-problems that can be directly solved. \cite{huang2021solving} used a clustering method to gather delivery tasks by spatiotemporal information; more specifically, tasks in the same hour in the same (pre-defined) region are categorized in the same cluster.
In the study of \cite{elsokkary2023crowdsourced}, delivery tasks are clustered based on the Euclidean distance between the geographical drop-off locations, by the k-medoids clustering algorithm. The same crowdsourcing worker delivers the set of tasks in each cluster, and a dedicated truck is operated to relay the corresponding parcels to the workers.
\cite{simoni2023crowdsourced} proposed an order batching algorithm for crowdsourced on-demand food delivery. The authors applied the divisive analysis clustering method that recursively splits clusters so that the decomposed sub-problems only contain orders that are likely to be served by the same courier. However, these clustering methods are often applied independently of solving sub-problems and do not achieve global matching efficiency.

\cite{wang2016towards} proposed an efficient en-route matching algorithm for a large-scale CSD problem. The authors considered utilizing crowdsourcing workers to fulfill last-leg delivery tasks from pick-own-parcel (pop) stations to customers. The matching problem between tasks and drivers was reduced to a network minimum-cost flow problem, where an edge represents a potential task-worker match. Combined with several network pruning strategies, they showed that their algorithm can be applied to large-scale problems. However, in their setting, the tasks and drivers are divided into pre-defined Voronoi cells based on pop stations, and the presented experiments focus on a single cell, i.e., cases with only a pick-up location. \cite{soto2017matching} studied a similar problem to \cite{wang2016towards} with an extended setting, where trip origins are not restricted to a single node, and time window constraints for both trips and requests are considered. The formulation is also presented as a minimum-cost flow problem; however, the authors do not discuss or show a result on the computational efficiency of the algorithm. As such, large-scale task decomposition for cases with different pick-up and drop-off locations or for achieving global matching efficiency remains unaddressed in the literature.

\subsection{Pricing and compensation schemes}
Many studies assume that a CSD platform determines compensation paid to drivers based on the characteristics of delivery tasks or drivers' detours from pre-planned trips required for performing the tasks \citep[e.g.,][]{archetti2016vehicle, wang2016towards}. However, this compensation scheme does not consider drivers' preferences and their heterogeneity. The pricing and compensation schemes can have a significant impact on drivers' willingness to work, which in turn influences the total operation cost of the CSD system \citep{gdowska2018stochastic}.

\cite{le2021designing} presented pricing and compensation schemes, based on the willingness to pay (WTP), the maximum price a customer is willing to pay for a delivery request, and the expectation to be paid (ETP), the driver's expected compensation for the additional cost spending in a detour from the pre-planned trip. The values of customers' WTP and drivers' ETP used in their study were generated based on the results from their previous empirical studies \citep{le2019influencing, le2019modeling}. In several demand and supply scenarios, \cite{le2021designing} tested different compensation schemes and found that individual-based price and compensation create more matches, leading to the highest profit for the CSD service provider. 

\cite{kafle2017design} studied a CSD system where crowdsourcing cyclists and pedestrians relay parcels with a truck carrier and fulfill tasks of the last-leg parcel delivery and the first-leg parcel pickup. The selection of CSD workers and the compensation paid to them are decided by bidding. The decision of the truck carrier on bid selection as well as route and schedule design is formulated as a mixed integer non-linear programming problem. \cite{allahviranloo2019dynamic} studied the impacts of CSD on the travel behavior of system participants by modeling a matching problem based on the WTP of customers and the ETP of drivers, where an auction is considered as the matching rule. Although these studies incorporate a bid-based matching concept into CSD, they do not discuss a mechanism design of the matching market, and the bid prices are given arbitrarily and independently of delivery costs that CSD workers experience.  

\section{Matching problem of crowdsourced delivery}\label{sec:problem}
This section defines and formulates the matching problem for the CSD system. We formulate the matching problem as a combinatorial optimization problem to achieve the socially optimal state, and show that its optimality conditions can be interpreted as market equilibrium conditions. The major challenges in solving the matching problem are also discussed.

\subsection{Problem definition}
Let $\cG\equiv (\cN,\cL)$ be a directed graph representing a transportation network where $\cN$ and $\cL$ are the sets of nodes and links, respectively, and each link $(i,j)\in\cL$ is associated with the link travel time $t_{ij}$. This study assumes that the travel time is static and deterministic, and the travels of CSD participants do not affect it.
On this network, we consider a CSD system illustrated by the conceptual diagram in \Cref{fig:system}. The main components of the system are defined and assumed as follows.

\begin{itemize}
    \item \textbf{Delivery task}. A delivery task is a unit of delivery operations, and it must be delivered from pickup to delivery locations $(r, s) \in \cT \subseteq\cN\times\cN$ on the network.
    We suppose that all delivery tasks in each pickup-delivery pair are assumed to be homogeneous and it is sufficient to distinguish them by their pickup-delivery pair $(r,s)$. This can be realized by grouping delivery operations into delivery-task units so that every delivery task is as homogeneous as possible. The number of tasks delivered between $(r,s)$ is denoted by $n_{rs}$.
    \item \textbf{CSD driver}. A CSD driver is an ordinary network traveler who plans to travel between an origin-destination (OD) pair $(o,d) \in \cW \subseteq \cN\times\cN$ for his/her own purpose and performs a delivery task by making a detour. For instance, an origin and a destination can be a residential zone and a central business district. Each driver $a \in \cA$ decides whether to perform a task for $(r,s)$ considering the detour disutility $c^{rs}_a$, which is heterogenous among drivers, as well as the reward $w_{rs}$ for performing the task. As such, a CSD driver tries to maximize his/her utility $w_{rs} - c^{rs}_a$. In this study, we assume that each driver performs a unit of delivery task for a single pair of pickup-delivery locations. 
    \item \textbf{System manager}. The system manager is an agent who receives delivery requests from customers and outsources the delivery tasks to CSD drivers. The manager aims to efficiently match tasks to drivers so as to maximize social surplus, which is the sum of the system profit and drivers' utilities. The revenue from the customers is assumed to be fixed in this study so that the profit can be evaluated only on the rewards to be paid to drivers. We also consider the cost $\bar{c}_{rs}$ of a regular driver the manager employs to perform a task for $(r,s)$ as the upper bound of the reward $w_{rs}$ for outsourcing. 
    \item \textbf{Matching market}. A matching market is opened as an auction by the system manager for efficient task-driver matching. Each CSD driver submits bids based on his/her private utility for performing tasks, based on which the assignment and reward $w_{rs}$ for a task are determined. 
\end{itemize}

Note that the single unit delivery per driver is a strong assumption and ignores the possiblity of a driver performing tasks with multiple times of pick-ups \citep[as in][]{archetti2016vehicle, soto2017matching,
dayarian2020crowdshipping}, but this is more applicable for en-route matching \citep{alnaggar2021crowdsourced}. 
The simplicity also allows for a clear presentation of our theoretical framework, which is the main focus of this paper. However, the fluid-particle decomposition approach of this paper can be extended to the case with multiple tasks per driver and/or price elasticity, and we leave the extension for future study (also see \Cref{sec:conclusion}).

\subsection{System optimal matching problem}
Next, we formulate a system optimal matching problem. The decision variable of the problem is a task-driver matching pattern $\mathbf{y} = [y^{rs}_{a} \in \{0, 1\}]_{a \in \cA, rs \in \cT}$, where $y^{rs}_{a}$ takes $1$ if driver $a$ is matched delivery task $(r,s)$ and $0$ otherwise. 
The social surplus associated with a matching pattern $\mathbf{y}$ is the sum of the system profit and the total utilities of CSD drivers. The system profit considers a fixed revenue $R$, rewards paid to CSD drivers $w_{rs}\sum_{a\in\cA} y^{rs}_{a}$, and the cost to operate dedicated vehicles $\overline{c}_{rs}\qty(n_{rs} - \sum_{a\in\cA} y^{rs}_{a})$, resulting in
\begin{align}\nonumber
    R - w_{rs}\sum_{a\in\cA} y^{rs}_{a} - \overline{c}_{rs}\qty(n_{rs} - \sum_{a\in\cA} y^{rs}_{a}).
\end{align}
The utility of a CSD driver gains is the reward minus the detour disutility $(w_{rs} -c^{rs}_{a})y^{rs}_{a}$, thus by summing up over all drivers the total utility is
\begin{align}\nonumber
    \sum_{a\in\cA} (w_{rs} -c^{rs}_{a})y^{rs}_{a}.
\end{align}
The sum of these two objectives yields $R - n_{rs}\overline{c}_{rs} + \sum_{a\in\cA} (\overline{c}_{rs} -c^{rs}_{a})y^{rs}_{a}$. 
Since $R - n_{rs}\overline{c}_{rs}$ is constant, the objective of the system optimal matching problem to be maximized reduces to
\begin{align}
    \label{eq:obj}
    \sum_{a\in\cA} (\overline{c}_{rs} -c^{rs}_{a})y^{rs}_{a},
\end{align}
which can be interpreted as the ``cost savings'' that society obtains by outsourcing delivery tasks to CSD drivers. 
As a result, the optimization problem [SO-P] for achieving the system optimal matching of delivery tasks and CSD drivers is formulated as follows.
\begin{align}
    [\text{SO-P}]\qquad
    \max_{\mathbf{y}}\quad&\sum_{a\in\cA}\sum_{rs \in \cT} (\overline{c}_{rs} - c^{rs}_{a}) y^{rs}_a \label{eq:SO_obj}\\
    \subto\quad&\sum_{rs \in \cT} y^{rs}_{a} = 1 \qquad \forall a\in\cA,\label{eq:indv_driver_consv}\\
    &\sum_{a\in\cA}y^{rs}_{a} \leq n_{rs} \qquad \forall(r,s)\in\cT,\label{eq:indiv_task_consv}\\
    &y^{rs}_{a} \in \{0,1\} \qquad \forall a\in\cA, \, \forall rs \in\cT,\label{eq:y_binary}
\end{align}
where the first constraint \eqref{eq:indv_driver_consv} is the condition that each driver performs one delivery task, and the second constraint \eqref{eq:indiv_task_consv} states that the supply of drivers must be equal to or less than the number of delivery tasks. 

\subsection{Optimality conditions and market equilibrium}
Due to the 0--1 binary constraints \eqref{eq:y_binary}, the problem [SO-P] is a combinatorial optimization problem. 
However, because the constraint matrices are totally unimodular, we can relax the problem [SO-P] as linear programming by replacing the binary constraints \eqref{eq:y_binary} with non-negative constraints $y^{rs}_{a}\geq 0$. Based on the linearly relaxed formulation, we can analyze the optimal state of the matching problem [SO-P] as follows. 

We define the Lagrangian $\cL$ as
\begin{align}
    \cL(\mathbf{y},\bm{\lambda},\bm{\rho}) 
    &\coloneqq \sum_{a\in\cA}\sum_{rs \in \cT} (\overline{c}_{rs} - c^{rs}_{a})y^{rs}_a 
    + \sum_{rs \in \cT} \lambda_{rs} \qty(n_{rs} - \sum_{a\in\cA}y^{rs}_{a}) + \sum_{a\in\cA} \rho_{a} \qty(1 - \sum_{rs \in \cT} y^{rs}_{a}) \\
    &= \sum_{a\in\cA}\sum_{rs \in \cT} (\overline{c}_{rs} - c^{rs}_{a} - \lambda_{rs} - \rho_{a}) y^{rs}_{a} + \sum_{rs \in \cT} n_{rs}\lambda_{rs} + \sum_{a\in\cA} \rho_{a}.
\end{align}
where $\bm{\rho} \geq \bm{0}$ and $\bm{\lambda} \geq \bm{0}$ are the Lagrangian multipliers associated with the constraints \eqref{eq:indv_driver_consv} and \eqref{eq:indiv_task_consv}, respectively. The optimality conditions of [SO-P] are
\begin{align}\label{eq:opt_cond_sop}
    \begin{dcases}
        \pdv{\cL}{y^{rs}_{a}} = \overline{c}_{rs} - c^{rs}_{a} - \lambda_{rs} - \rho_{a} \leq 0 ,\, y^{rs}_{a} \geq 0 ,\, y^{rs}_{a} \pdv{\cL}{y^{rs}_{a}} = 0, \\
        \pdv{\cL}{\lambda_{rs}} = n_{rs} - \sum_{a\in\cA}y^{rs}_{a} \geq 0,\, \lambda_{rs} \geq 0,\, \lambda_{rs}\qty(n_{rs} - \sum_{a\in\cA}y^{rs}_{a}) = 0, \\
        \pdv{\cL}{\rho_{a}} = 1 - \sum_{rs \in \cT} y^{rs}_{a} = 0.
    \end{dcases}
\end{align}

The conditions \eqref{eq:opt_cond_sop} result in the following three conditions:
\begin{align}\label{eq:utility_max}
    &(\text{Utility maximization})\qquad
    \begin{dcases}
        w^\star_{rs} - c^{rs}_{a} = \rho^\star_{a} & \text{if $y^{rs\star}_{a} > 0$} \\
        w^\star_{rs} - c^{rs}_{a} \leq \rho^\star_{a} & \text{if $y^{rs\star}_{a} = 0$}
    \end{dcases}\qquad \forall a\in\cA,\, \forall(r,s)\in\cT, \\
    \label{eq:supply_and_demand}
    &(\text{Supply and demand})\qquad
    \begin{dcases}
        \sum_{a\in\cA} y^{rs\star}_{a} = n_{rs} & \text{if $\overline{c}_{rs} - w^\star_{rs} > 0$} \\
        \sum_{a\in\cA} y^{rs\star}_{a} \leq n_{rs} & \text{if $\overline{c}_{rs} - w^\star_{rs} = 0$}
    \end{dcases}\qquad \forall rs \in\cT, \\
    \label{eq:conservation}
    &(\text{Conservation})\qquad \qquad
    \sum_{rs \in \cT} y^{rs\star}_a = 1 \qquad \forall a\in\cA.
\end{align}
where $w^\star_{rs}$ is the reward for a task for $(r,s)$ at optimum, which is given by
\begin{align}
    w^\star_{rs} \equiv \overline{c}_{rs} - \lambda^\star_{rs}.
\end{align}
That is, the optimal reward $w^\star_{rs}$ for each pickup-location pair $(r,s)$ can be calculated through the Lagrangian multiplier $\lambda^\star_{rs}$ with respect to the constraint \eqref{eq:indiv_task_consv}.

The first condition \eqref{eq:utility_max} describes the utility maximization principle of drivers. In other words, if $a$ is matched to a task for $(r,s)$, then performing the task maximizes the driver's utility $\rho^\star_{a} = \max_{rs \in \cT}\qty{w^\star_{rs} - c^{rs}_{a}}$. The second supply-demand equilibrium condition \eqref{eq:supply_and_demand} states that at market equilibrium, the number of drivers performing tasks for $(r,s)$ is equivalent to the number of tasks if the reward $w^\star_{rs}$ is less than the operation cost $\overline{c}_{rs}$, and the number of drivers is in short supply otherwise and the manager operates the dedicated vehicle to complete the task. Finally, the condition \eqref{eq:conservation} simply states that drivers who express their participation in the CSD system have to be matched with exactly one delivery task. 

That being said, the optimality conditions of the problem [SO-P] describe the equilibrium under a perfectly competitive market: (i) the maximization of utility for CSD drivers, (ii) the supply and demand equilibrium condition, and (iii) the conservation condition. As such, the following proposition with the definition of equilibrium in the CSD system holds.
\begin{definition}\label{def:equi}
(Equilibrium matching pattern).
An equilibrium matching pattern under a perfectly competitive market is a tuple of state variables $(\mathbf{y}^\star, \bm{w}^\star, \bm{\rho}^\star)$ that satisfies the conditions \eqref{eq:utility_max}, \eqref{eq:supply_and_demand}, and \eqref{eq:conservation}.
\end{definition}
\begin{proposition}\label{prop:1} 
An equilibrium matching pattern realized under the CSD system maximizes the social surplus defined by \eqref{eq:SO_obj}.
\end{proposition}

\subsection{Challenges in solving the matching problem}

The above theoretical analysis revealed that if we find an equilibrium matching pattern, then it also achieves the system's optimum. However, achieving the equilibrium is non-trivial and requires a mechanism design so that drivers have no incentive to take strategic action affecting the market prices of delivery tasks. Although the system manager has to solve the problem [SO-P] to this end, the problem cannot be directly solved due to the following two reasons.

Firstly, as the objective function \eqref{eq:obj} of the problem [SO-P] is the function of the drivers' utility that is generally unobservable, it cannot be directly evaluated (\textbf{Issue 1}). Because CSD drivers are not dedicated to delivery and travel anyway for their own purposes, their preferences for delivery tasks can be heterogeneous. In other words, the utility of drivers for performing a delivery task by taking detours is private information that cannot be generally observed.
Secondly, even if the drivers' utility is observed, the problem [SO-P] is a large-scale combinatorial optimization problem with $|\cA|\times|\cT|$ number of decision variables and cannot be solved in reasonable computational time  (\textbf{Issue 2}). Consider drivers and delivery tasks are distributed evenly on the network. Then, the size of solution space is $O(|\cN|^4)$, and even with several hundreds of nodes, the problem would have more than one billion decision variables. 
In addition, for the mechanism design such as the VCG mechanism, the problem [SO-P] must be solved repeatedly many times, which is computationally impossible. 

Due to these two challenges, it is not possible to achieve the system optimal matching by directly solving the problem [SO-P]. Therefore, it is necessary to develop a mechanism of matching market design that is sufficiently efficient and can observe the private information on drivers' utility so that the manager can achieve socially optimal matching.


%
\section{Matching mechanism design}\label{sec:design} 
To address the two challenges \textbf{Issue 1} and \textbf{Issue 2}, we decompose the problem [SO-P] into a two-stage hierarchical problem consisting of the task partition (\textit{master problem}) and individual task--driver matching within smaller groups of drivers (\textit{sub-problems}), as illustrated in \Cref{fig:approach}. 
The sub-problem is a smaller-scale matching problem for each group of drivers,  and we can apply an auction mechanism to solve it (\Cref{sec:auction}). This allows the system manager to observe the private information on each driver's utility through the submitted bids, thereby overcoming \textbf{Issue 1}. We also formulate the master problem as a fluidly approximated problem based on random utility theory (\Cref{sec:fluid}). Given that the distribution of drivers' utility is estimated using the information from past auctions, the optimal value functions of the sub-problems can be analytically derived, which allows for the evaluation of the objective function of the master problem without solving individual matching problems. Furthermore, the decision variable of the master problem is aggregated with its size being $|\cW| \times |\cT|$, which is significantly smaller than the original size. Together with the fact that the sub-problems are independent of each other and can be solved in parallel, our approach efficiently solves the matching problem, thus addressing \textbf{Issue 2} (the algorithm is described in detail in \Cref{sec:algorithm}). As such, our approach overcomes both \textbf{Issue 1} and \textbf{Issue 2} simultaneously.

%
\subsection{Hierarchical market decomposition}\label{subsec:decompose}
We introduce an aggregate variable $\mathbf{f} = [f^{rs}_{od}]_{rs \in \cT, od \in \cW}$ representing the number of delivery tasks $(r,s)$ allocated to a set of drivers $\cA_{od} \subseteq \cA$ whose OD pair is $(o,d)$.
Then, the problem [SO-P] can be equivalently formulated as the following integer programming problem:
\begin{align}
    [\text{SO-P-2}]\qquad
    \max_{\mathbf{y},\,\mathbf{f}}\quad&\sum_{od\in\cW}\qty[\sum_{a\in\cA_{od}}\sum_{rs \in \cT} (\overline{c}_{rs} - c^{rs}_{a})y^{rs}_a] \label{eq:SO2_obj}\\
    \subto\quad
    &\sum_{od \in\cW} f^{rs}_{od} \leq n_{rs} \qquad \forall rs \in\cT,\label{eq:agg_task_consv}\\
    &\sum_{a\in\cA_{od}}y^{rs}_{a} = f^{rs}_{od} \qquad \forall od\in\cW, \, \forall rs\in\cT,\label{eq:SO2_con_3}\\
    &f^{rs}_{od} \in \Z_{+} \qquad \forall od\in\cW, \,\forall rs\in\cT, \label{eq:f_integer}\\
    & \eqref{eq:indv_driver_consv}, \eqref{eq:y_binary}, \nonumber
\end{align}
where $\Z_{+}$ is the set of all positive integers. 
The original supply-demand constraint \eqref{eq:indiv_task_consv} is replaced with constraints \eqref{eq:agg_task_consv} and \eqref{eq:SO2_con_3} using the intermediate variable $\mathbf{f}$.

Then, we can hierarchically decompose the problem [SO-P-2] into a master problem and sub-problems as follows:
\begin{align}
    [\text{SO-P/Master}]\qquad&\nonumber\\\max_{\mathbf{f}}\quad&\sum_{od\in\cW}z^{\star}_{od}(\mathbf{f})  \label{eq:SO_mas_obj}\\
    \subto\quad&\sum_{rs \in \cT} f^{rs}_{od} = q_{od} \qquad \forall od\in\cW,\label{eq:agg_driver_consv}\\
    &\eqref{eq:agg_task_consv},\, \eqref{eq:f_integer}, \nonumber \\
    \text{where}\quad&
    z^{\star}_{od}(\mathbf{f}) \equiv \max_{\mathbf{y}_{(od)}} z_{od}(\mathbf{y}_{(od)}|\mathbf{f}) \label{eq:optimal_value_func}\\
    & q_{od} \equiv |\cA_{od}| \\
    [\text{SO-P/Sub$(od)$}]\qquad&\nonumber\\
    \max_{\mathbf{y}_{(od)}} \quad
    & z_{od}(\mathbf{y}_{(od)}|\mathbf{f}) \equiv \sum_{a\in\cA_{od}}\sum_{rs \in \cT} (\overline{c}_{rs} - c^{rs}_{a})y^{rs}_a \label{eq:SO_sub_obj}\\
    \subto\quad&\sum_{rs \in \cT} y^{rs}_{a} = 1 \qquad \forall a\in\cA_{od},\label{eq:indiv_driver_consv_od}\\
    &\sum_{a\in\cA_{od}}y^{rs}_{a} = f^{rs}_{od} \qquad \forall(r,s)\in\cT,\label{eq:indiv_task_consv_od}\\
    &y^{rs}_{a} \in \{0,1\} \qquad \forall a\in\cA_{od}, \, \forall rs \in\cT,\label{eq:y_binary_od}. 
\end{align}
Here we introduced $q_{od}$ representing the total number of drivers whose OD pair is $(o,d)$, and the constraints \eqref{eq:indiv_driver_consv_od} and \eqref{eq:y_binary_od} are restricted to $\cA_{od}$.
The objective function of the master problem \eqref{eq:SO_mas_obj} is defined as the sum of optimal value functions of [SO-P/Sub$(od)$] represented by \eqref{eq:SO_sub_obj}.

We now have a master problem and sub-problems for the original problem [SO-P], which are much smaller in scale than the original problem. However, both of the problems are still nontrivial to solve. Although the sub-problem is a small-scale combinatorial optimization problem, it also involves the private information $c^{rs}_{a}$ that cannot be directly observed by the system manager. In addition, because the master problem is a task partition problem and determines the number of delivery tasks allocated to each group of drivers, it must be solved before the sub-problems. Nevertheless, the objective function of the master problem \eqref{eq:SO_mas_obj} is expressed using the optimal value function \eqref{eq:optimal_value_func} of the sub-problem, which is supposed to be obtained by solving each sub-problem.
Our approach based on the fluid-particle decomposition addresses both of these problems, as presented in the following sections.
\subsection{Auction mechanism for particle matching}\label{sec:auction}
The sub-problem [SO-P/Sub$(od)$] matches CSD drivers and delivery tasks in a particle unit, given the partitioned tasks $\mathbf{f}_{od}$, so as to maximize the social surplus for each driver group $od \in \cW$. 
To solve the sub-problem, it is necessary to know each driver's utility $c^{rs}_{a}$ for performing a task, which is originally private information.
To this end, we implement an auction mechanism for each group of drivers $\cA_{od}$.
Introducing an auction market allows the system manager to directly observe the private valuations of drivers for performing delivery tasks through their bids.

To achieve the socially optimal state, an auction mechanism must be designed so that bidders (i.e., drivers) who participate in the auction have no incentive to make a strategic action.
Otherwise, the resultant delivery task assignment and rewards would not satisfy the market equilibrium conditions, thus not the system's optimum.
As such, the auction mechanism for the sub-problem [SO-P/Sub($od$)] must have the following two properties: (i) \textit{truth-telling}: bidders have no incentive to make false bids; and (ii) \textit{efficiency}: the matching pattern realized by the auction coincides with the solution of the problem [SO-P/Sub($od$)].

The VCG market mechanism \citep{vickrey1961counterspeculation, clarke1971multipart, groves1973incentives} is known as a mechanism that satisfies the two properties, and we apply it to our sub-problems. Specifically, we design the VCG mechanism for the CSD system as follows.

\renewcommand{\theenumi}{\arabic{enumi}}
\renewcommand{\labelenumi}{\textit{Step \theenumi}:}
\renewcommand{\theenumii}{\arabic{enumii}}
\renewcommand{\labelenumii}{\textit{Step \theenumi-\theenumii}:}
\begin{enumerate}
    \setlength{\leftskip}{0.5cm}
    \item Each CSD driver $a \in \cA_{od}$ submits a bid $b^{rs}_{a}$ for a delivery task for $(r,s)$, indicating his/her cost to perform the task.
    \item The system manager allocates delivery tasks to each driver to maximize the social surplus with respect to the declared costs, by solving [SO-P/Sub($od$)] with $c^{rs}_{a}$ in \eqref{eq:SO_sub_obj} replaced with $b^{rs}_{a}$: 
    \begin{align}
        \max_{\mathbf{y}_{(od)}}\quad&
        Z(\mathbf{y}_{(od)}|\mathbf{b},\mathbf{f}) \equiv \sum_{a\in\cA_{od}}\sum_{rs \in \cT} (\overline{c}_{rs} - b^{rs}_{a})y^{rs}_a \label{eq:SO_sub_obj_bid}\\
        \subto\quad&\eqref{eq:indiv_driver_consv_od},\, \eqref{eq:indiv_task_consv_od},\, \eqref{eq:y_binary_od}. \nonumber
\end{align}
    \item The reward $w_a$ that CSD driver $a\in\cA_{od}$ receives is determined by
    \begin{align}
        \label{eq:vcg_reward}
        w_{a}(\mathbf{b}_a, \mathbf{b}_{-a}) &= \qty(Z^\star(\mathbf{b}_a, \mathbf{b}_{-a}) + \sum_{rs\in\cT} b^{rs}_{a} y^{rs\star}_{a}) - \max_{\mathbf{y}_{-a}} Z_{-a}(\mathbf{y}_{-a}|\mathbf{b}_{-a}) \nonumber\\
        &= \qty(Z^\star_{-a}(\mathbf{b}_a, \mathbf{b}_{-a}) + \sum_{rs\in\cT} \overline{c}_{rs} y^{rs\star}_{a}) - \max_{\mathbf{y}_{-a}} Z_{-a}(\mathbf{y}_{-a}|\mathbf{b}_{-a}),
    \end{align}
    where $Z^\star(\mathbf{b}) \equiv  \max_{\mathbf{y}_{(od)}} Z(\mathbf{y}_{(od)}|\mathbf{b},\mathbf{f})$ and $\mathbf{y}^{\star} = [y^{rs\star}_{a}]_{rs\in\cT, a\in\cA_{od}}$ are the optimal value of the objective and the optimal matching pattern for the problem in Step 2, respectively.
    The subscript ``$-a$'' denotes the set of all drivers except for driver $a$, i.e., $\cA_{od}\setminus\{a\}$. Therefore, $\mathbf{b}_{-a}$ represents the vector of bids excluding for $a$, $Z^\star_{-a}(\cdot)$ denotes the maximum objective value minus the contribution of $a$, and $\max_{\mathbf{y}_{-a}} Z_{-a}(\cdot)$ is the objective value achieved without participation of $a$.
\end{enumerate}
Herein, Eq. \eqref{eq:vcg_reward} implies that each driver is rewarded based on the loss that the society would incur without the driver, and the reward would be zero if the driver did not contribute to the system. As such, whereas most existing CSD matching systems compensate drivers according to their detour costs \citep{archetti2016vehicle, wang2016towards}, our mechanism determines rewards based on the drivers' contributions to society.

Moreover, the proposed VCG mechanism for the CSD system holds the following preferable properties.

\begin{proposition}\label{prop:auction}
    The VCG mechanism for the CSD system achieves truth-telling and efficient matching.
\end{proposition}
\begin{proof}
    \Cref{app:VCG}.
\end{proof}

The fact that the mechanism holds truth-telling property ensures that the bid price always represents the driver's actual perceived cost. That is, the mechanism allows the system manager to directly observe the private valuation of drivers through their bids, and by repeating the auctions on a daily basis, sufficient information will be obtained to estimate the distribution of drivers' utility. 
The system manager then can solve sub-problems replacing $c^{rs}_{a}$ in \eqref{eq:SO_sub_obj} with drivers' bids for the delivery tasks, which addresses \textbf{Issue 1}.

\subsection{Task partition based on fluid approximation}\label{sec:fluid}
The master problem of the proposed mechanism performs task partition and finds the number of delivery tasks to be allocated to each group of drivers.
However, as expressed in \eqref{eq:SO_mas_obj}, solving the master problem requires the evaluation of the optimal value function of the sub-problems $z^{\star}_{od}(\mathbf{f})$ without solving the sub-problems directly.
The reason for this is that the sub-problems are formulated with the task allocation $\mathbf{f}$ determined by the master problem as a given, and the sub-problems cannot be solved before the master problem is solved. 

To address this issue, we propose a fluid-approximation approach. The core idea of the approach is evaluating $z^{\star}_{od}(\mathbf{f})$ by considering the continuous distribution of drivers' utility $c^{rs}_{a}$ for performing delivery tasks. The rationale behind this is that if the auction described in \Cref{sec:auction} is conducted on a daily basis, the system manager can collect the information on drivers' utility sufficiently to estimate the distribution, even though the utility of the drivers who perform tasks on the day is still unknown. Specifically, we rely on an additive random utility model (ARUM) framework\footnote{While this paper focuses on the ARUM framework, the proposed approach can be applied to other RUMs such as weibit or q-product random utility models \citep[e.g.,][]{chikaraishi2016discrete} by deriving the expected maximum utility and its convex conjugate function.} as follows.

\begin{assumption}\label{assumption:utility}
    There are a sufficiently large number of drivers for each OD pair $od \in \cW$. Then, the disutility $c^{rs}_{a}$ of driver $a \in \cA_{od}$ to perform a task for $(r,s)$ can be approximated by
    \begin{align}
        \label{eq:ARUM_disutility}
        c^{rs}_{a} \thickapprox C^{rs}_{od} - \epsilon^{rs}_{a} \qquad \forall a \in \cA_{od}, \forall od \in\cW, \, \forall rs \in\cT,
    \end{align}
    where $C^{rs}_{od}$ is the deterministic detour disutility for a driver who travels between OD pair $(o,d)$ and performs a task for $(r,s)$, which the system manager can observe, and $\varepsilon^{rs}_{a}$ is the unobservable random utility following a joint distribution with finite means that is continuous and independent of $C^{rs}_{od}$. 
\end{assumption}
Note that the deterministic detour cost $C^{rs}_{od}$ is typically defined as
\begin{align}
    \label{eq:detour_cost}
    C^{rs}_{od} \equiv t_{or} + t_{rs} + t_{sd} - t_{od},
\end{align}
where $t_{nm}$ is the travel time of the shortest path between node pair $(n, m)$ (also see \Cref{fig:system}).

Our objective is to analytically derive the optimal value function of the problem [SO-P/Sub$(od)$] that leads to the fluidly approximated formulation of the master problem [SO-P/Master]. To this end, we begin by analyzing the Lagrangian dual problem of [SO-P/Sub$(od)$].

\begin{lemma}\label{lemma:so_d_sub}
    The Lagrangian dual problem of [SO-P/Sub$(od)$] is defined as follows.
    \begin{align}
    &[\text{SO-D/Sub$(od)$}]\nonumber\\
    &\qquad z^{\star}_{od}(\mathbf{f}) = \min_{\bm{\sigma}} \sum_{a\in\cA_{od}} \max_{rs \in \cT} \qty{\overline{c}_{rs} - c^{rs}_{a} - \sigma^{rs}_{od}} + \sum_{rs \in \cT} f^{rs}_{od} \sigma^{rs}_{od}
    \end{align}
\end{lemma}
\begin{proof}
    \Cref{app:proof_so_d_sub}.
\end{proof}

Under \Cref{assumption:utility}, the first term of the objective function of [SO-D/Sub$(od)$] can be approximated by the \textit{expected maximum utility}, or surplus, $S_{od}$ for the CSD drivers who travel between $(o,d)$:
\begin{align}
    \sum_{a\in\cA_{od}} \max_{rs \in \cT} \qty{\overline{c}_{rs} - c^{rs}_{a} - \sigma^{rs}_{od}} \thickapprox q_{od} \cdot S_{od}(\mathbf{v}_{od})
\end{align}
where
\begin{align}
    S_{od}(\mathbf{v}_{od}) = 
    \mathbb{E}\qty[\max_{rs \in \cT}\qty{v^{rs}_{od} + \epsilon^{rs}_{a}}],
\end{align}
and $v^{rs}_{od} = \overline{c}_{rs} - \sigma^{rs}_{od} - C^{rs}_{od}$.
As a result, the problem [SO-D/Sub$(od)$] reduces to:
\begin{align}
    z^{\star}_{od}(\mathbf{f}) &\approx \min_{\mathbf{v}_{od}} \left[ \sum_{rs \in \cT} f^{rs}_{od} (\overline{c}_{rs} - C^{rs}_{od} - v^{rs}_{od}) + q_{od} \cdot S_{od} (\mathbf{v}_{od}) \right] \nonumber\\
    &= \max_{\mathbf{v}_{od}} \left[ \sum_{rs \in \cT} f^{rs}_{od} v^{rs}_{od} - q_{od} \cdot S_{od} (\mathbf{v}_{od}) \right]
    - \sum_{rs \in \cT} f^{rs}_{od} (\overline{c}_{rs} - C^{rs}_{od}).
    \label{eq:sod_sub} 
\end{align}
Furthermore, it is known for any ARUM discrete choice model that the expected maximum utility function $S$ holds the convex conjugate duality. That is, there exists a \textit{convex conjugate function} $H$ satisfying
\begin{align}
    \label{eq:conjugate}
    H_{od}(\mathbf{p}_{od}) = \max_{\mathbf{v}_{od}} \qty{\mathbf{p}_{od} \cdot \mathbf{v}_{od} - S_{od}(\mathbf{v}_{od})},
\end{align}
where $\mathbf{p}_{od} = [p^{rs}_{od} \equiv f^{od}_{rs}/q_{od}]_{od \in \cW, rs \in \cT}$ is the vector of the allocation ratios of the tasks to the driver group $(o,d)$.
Note that \cite{fosgerau2020discrete} called the negative convex conjugate $-H$ a \textit{generalized entropy}\footnote{The properties of Legendre transformation are presented in detail in \cite{sandholm2010population}. \cite{oyama2022markovian} and \cite{akamatsu2023global} also discussed the conjugate function of the equilibrium assignment based on a network generalized extreme value model, which is a general class of ARUM discrete choice model by \cite{Daly2006NGEV}.}.
By multiplying the both sides of \eqref{eq:conjugate} by $q_{od}$, we obtain
\begin{align}
    \label{eq:sum_conjugate}
    \hat{H}_{od}(\mathbf{f}) \equiv q_{od} H_{od}(\mathbf{p}) &= 
    \max_{\mathbf{v}_{od}} \left[ \sum_{rs \in \cT} f^{rs}_{od} v^{rs}_{od} - q_{od} S_{od} (\mathbf{v}_{od}) \right].
\end{align}

Therefore, we approximately obtain the optimal value function of the sub-problem [SO-P/Sub$(od)$] with any ARUM model with the generalized entropy, and the following lemma holds.
\begin{lemma}\label{lemma:z}
Suppose that \Cref{assumption:utility} holds.
Then, the optimal value function $z^{\star}_{od}(\mathbf{f})$ in \eqref{eq:SO_mas_obj} can be expressed in closed-form as
\begin{align}
    \label{eq:opt_val_func}
    z^{\star}_{od}(\mathbf{f}) \thickapprox  \sum_{rs \in \cT} f^{rs}_{od} (\overline{c}_{rs} - C^{rs}_{od}) - \hat{H}_{od}(\mathbf{f}),
\end{align}
where $-H_{od}$ is the generalized entropy function, which is the negative convex conjugate of the expected maximum utility $S_{od}$ for any ARUM discrete choice model. 
\end{lemma}
Finally, by aggregating \eqref{eq:opt_val_func} over all OD pairs, we obtain the fluidly-approximated version of the master problem [SO-P/Master] as follows.
\begin{proposition}\label{prop:so_a_p}
    Under \Cref{assumption:utility}, the master problem [SO-P/Master] can be fluidly approximated as 
    \begin{align}
    [\text{SO-A-P}]\qquad&\nonumber\\
    \max_{\mathbf{f} \geq \bm{0}}\quad& \sum_{od \in \cW} \sum_{rs \in \cT} f^{rs}_{od} (\overline{c}_{rs} - C^{rs}_{od}) - \sum_{od \in \cW} \hat{H}_{od}(\mathbf{f}) \label{eq:SOAP_mas_obj}\\
    \subto\quad 
    &\eqref{eq:agg_task_consv}, \eqref{eq:agg_driver_consv}, \nonumber\\
    \text{where}\quad& \hat{H}_{od}(\mathbf{f}) = \max_{\mathbf{v}_{od}} \qty{\mathbf{f} \cdot \mathbf{v}_{od} - q_{od} S_{od}(\mathbf{v}_{od})}. \label{eq:entropy}
    \end{align}    
\end{proposition}

Note that the optimal solution of the fluidly-approximated problem [SO-A-P] is not exactly the same as that of the original problem [SO-P/Master], as the true distribution of the drivers' utility for performing delivery tasks differs from the estimated distribution.
Nevertheless, the bias can be mitigated by repeating auctions mentioned in \Cref{sec:auction} on a daily basis and updating the estimated distribution using the information collected through bids submitted by drivers.
In addition, the continuous relaxation of the integer constraints may cause an error in the objective functions. This error would become negligibly small in a large-scale problem, i.e., if the number of drivers $|\cA_{od}|$ for each OD pair is sufficiently large.
That is, the proposed fluid-particle decomposition approach allows the system manager to approximately but accurately determine the task partition pattern $\mathbf{f}^\star$ by solving the fluid-approximated problem [SO-A-P], instead of the original problem [SO-P/Master] that is intractable.

%
\section{Solution algorithm}\label{sec:algorithm}
This section presents a special but useful case where the drivers' unobserved utility for performing tasks follows the type I extreme value distribution, i.e., the logit model. We show that the problem [SO-A-P] can be viewed as an entropy-regularized optimal transport (EROT) problem and present its efficient solution algorithm. 

\subsection{Logit-based formulation as entropy-regularized optimal transport}
This section considers the specific case of the multinomial logit (MNL) model for the fluid approximation. Therefore, the unobserved random utility $\epsilon^{rs}_{a}$ of driver $a \in \cA_{od}$ performing a task follows a type I extreme value distribution with scale $\theta$. 
It is well known that the generalized entropy $-H$ for the MNL model with scale $\theta$ reduces to the Shannon's entropy as
\begin{align}
    \label{eq:shannon}
    -H_{od}(\mathbf{p}_{od}) = - \frac{1}{\theta} \sum_{rs \in T} p^{rs}_{od} \log p^{rs}_{od},
\end{align}
and thus,
\begin{align}
    \label{eq:shannon_f}
    -\hat{H}_{od}(\mathbf{f}_{od}) = - \frac{1}{\theta} \sum_{rs \in T} \sum_{rs \in T} f^{rs}_{od} \log \frac{f^{rs}_{od}}{q_{od}}.
\end{align}
As a result, the logit case of the problem [SO-A-P] is re-written as the following \textit{minimization} problem: 
\begin{align}
    [\text{SO-A-P/Logit}]\qquad&\nonumber\\
    \min_{\mathbf{f} \geq \bm{0}}\quad& \sum_{od \in \cW} \sum_{rs \in \cT}  f^{rs}_{od} ( C^{rs}_{od} - \overline{c}_{rs}) + \frac{1}{\theta} \sum_{od \in \cW} \sum_{rs \in T} f^{rs}_{od} \qty(\log \frac{f^{rs}_{od}}{q_{od}} - 1)
    \label{eq:SOAP_logit_obj}\\ 
    \subto \quad
    &\eqref{eq:agg_task_consv}, \eqref{eq:agg_driver_consv}, \nonumber
    %
\end{align}
Note that we added a constant $\sum_{od \cW} \sum_{rs \cT} f^{rs}_{od}/\theta = |\cA|/\theta$ for the simple design of algorithm (in \Cref{sec:master_algorithm}). 
From this reformulation, we obtain the following observation:
\begin{observation}
    The problem [SO-A-P/Logit] can be viewed as an \textit{entropy-regularized optimal transport} (EROT) problem, as it minimizes the extra transport cost with entropy regularization under supply and demand conservation constraints. This mathematical model is also known as a \textit{doubly-constrained gravity model} in the transportation research field.
\end{observation}
Note that as mentioned above, OT or doubly-constrained gravity models are classical models in the transportation research field. However, OT problems have been recently and extensively studied in the machine learning (ML) field \citep[e.g.,][]{peyre2019computational}, because the optimal transport cost between probability distributions is often used as a loss function. Moreover, EROT has been introduced by \cite{cuturi2013sinkhorn} into the ML community as a framework that has a number of computational advantages (e.g., convexity and differentiability) over the original OT and can be efficiently solved even in large-scale problems.


\subsection{Algorithm for the master problem}\label{sec:master_algorithm}
The master problem [\text{SO-A-P/Logit}] can be efficiently solved by Bregman's balancing method \citep{bregman1967relaxation, lamond1981bregman}, or Sinkhorn algorithm \citep{sinkhorn1967diagonal, sinkhorn1967concerning, knight2008sinkhorn}, which has been developed for solving the doubly-constrained gravity model \citep{wilson1973further, fisk1975note} and was recently introduced to the ML community by \cite{cuturi2013sinkhorn}.

For the application of the Sinkhorn algorithm to the problem [SO-A-P/Logit], we derive the following equation from its optimality condition with respect to $f^{rs}_{od}$:
\begin{align}
    \label{eq:opt_p}
    f^{rs}_{od} = q_{od} \exp{\theta \qty[(\overline{c}_{rs} - C^{rs}_{od}) + \rho_{od} + \lambda_{rs}]} 
    = q_{od} K^{rs}_{od} u_{od} v_{rs},
\end{align}
wherein we defined $\mathbf{K} = \exp(\theta \mathbf{C})$, $\mathbf{u} = \exp(\theta \bm{\rho})$, and $\mathbf{v} = \exp(\theta \bm{\lambda})$, and $\bm{\rho}$ and $\bm{\lambda}$ are the Lagrangian multipliers associated with the supply and demand constraints \eqref{eq:agg_driver_consv} and \eqref{eq:agg_task_consv}, respectively. 
From the other optimality conditions with respect to $\bm{\rho}$ and $\bm{\lambda}$, the relationships between $\mathbf{K}$, $\mathbf{u}$, and $\mathbf{v}$ are obtained as
\begin{align}
    & \mathbf{u} = \frac{\mathbf{1}_{|\cW|}}{\mathbf{K} \mathbf{v}} \label{eq:u_update}\\
    & \mathbf{v} \le \frac{\mathbf{n}}{\mathbf{K}\transpose \mathbf{u}}. \label{eq:v_update}
\end{align}
Using \eqref{eq:u_update} and \eqref{eq:v_update}, the Sinkhorn algorithm iteratively and alternately updates $\mathbf{u}$ and $\mathbf{v}$ until they converge, as follows.\\

\noindent [\textbf{Sinkhorn algorithm}]
\renewcommand{\theenumi}{\arabic{enumi}}
\renewcommand{\labelenumi}{\textit{Step \theenumi}:}
\renewcommand{\theenumii}{\arabic{enumii}}
\renewcommand{\labelenumii}{\textit{Step \theenumi-\theenumii}:}
\begin{enumerate}
    \setlength{\leftskip}{0.5cm}
    \item Initialize $\mathbf{u}^{(0)} := \mathbf{1}_{|\cW|}$ and $\mathbf{v}^{(0)} := \mathbf{1}_{|\cT|}$. Set $m := 1$.
    \item Update $\mathbf{u}^{(m)}$ by $\mathbf{u}^{(m)} := \frac{\mathbf{1}_{|\cW|}}{\mathbf{K} \mathbf{v}^{(m-1)}}$.
    \item Update $\mathbf{v}^{(m)}$ by $\mathbf{v}^{(m)} := \min \qty{ \frac{\mathbf{n}}{\mathbf{K}\transpose \mathbf{u}^{(m)}}, \mathbf{1}_{|\cT|}}$.
    \item Finish the algorithm if the convergence criteria are satisfied, and set $m := m + 1$ and go back to Step 2 otherwise.
\end{enumerate}

As is shown, this algorithm can be performed only with vector multiplications, and an iteration requires only $\cO(|\cW||\cT|)$ computation thus very efficient.
Note that the ``min'' operation in Step 3 corresponds to the inequality constraint of \eqref{eq:agg_task_consv}. 
After the algorithm converged, we obtain the optimal task partition pattern $\mathbf{f}^\star$ by \eqref{eq:opt_p} as
\begin{align}
    \mathbf{f}^\star = \text{diag}(\mathbf{q}) \qty[ \text{diag}(\mathbf{u}^\star) ~ \mathbf{K} ~ \text{diag}(\mathbf{v}^\star)].
\end{align}
Moreover, the optimal Lagrangian multipliers are
\begin{align}
    &\bm{\rho}^\star = \frac{1}{\theta} \log \mathbf{u}^\star, \\ 
    &\bm{\lambda}^\star = \frac{1}{\theta} \log \mathbf{v}^\star,
\end{align}
which yields the optimal reward $\mathbf{w}^\star = [w^\star_{rs}]$ at an aggregate level for each $rs \in \cT$:
\begin{align}
    \mathbf{w}^\star = \mathbf{\overline{c}} - \bm{\lambda}^\star.
\end{align}

\section{Numerical experiments}\label{sec:experiment}
This section presents numerical experiments to show the performance of the approach and algorithm proposed in this paper.
We compare our approach to the \textit{naive mechanism} that directly solves [SO-P] as a linear programming (LP) model, from the viewpoint of the computational efficiency and approximation errors, to see how efficient the proposed approach is and how accurately it approximates the original solution of [SO-P]. In this experiment, we consider an MNL model for the fluid approximation as described in \Cref{sec:algorithm}.

Throughout the experiments, we implemented the algorithms in the environment of Python 3.9 on a machine with 14-core Intel Xeon W processors (2.5 GHz) and 64 GB of RAM. We used Gurobi Optimizer version 10.0.1 with the default parameter settings to solve linear programs, and the other solution algorithms including the balancing method have been implemented by writing our own codes. The Sinkhorn algorithm for the problem [SO-A-P/Logit] was terminated when both of the maximum relative differences of $\mathbf{u}$ and $\mathbf{v}$ were smaller than $10^{-5}$.

\subsection{Problem settings}
The experiment used Winnipeg network data provided by \citet{trans}, which contains 1,052 nodes, 2,836 links, 147 zones, and 4,345 possible zone pairs for drivers' OD pairs or tasks' pickup-delivery locations.
Each dataset consisting of delivery tasks and possible CSD drivers was generated by the following procedure:
\renewcommand{\theenumi}{\arabic{enumi}}
\renewcommand{\labelenumi}{\textit{Step \theenumi}:}
\renewcommand{\theenumii}{\arabic{enumii}}
\renewcommand{\labelenumii}{\textit{Step \theenumi-\theenumii}:}
\begin{enumerate}
    \setlength{\leftskip}{0.5cm}
    \item Set parameters $|\cW|$, $|\cT|$, $|\cA|$, and $\theta$, wherein we assumed that the total number of tasks $N$ is double the number of drivers $2\cdot|\cA|$. 
    \item Randomly choose $|\cW|$ different OD pairs and $|\cT|$ different pickup-delivery location pairs from the Winnipeg network zone pairs.
    \item Randomly decide the OD pair of each of all $|\cA|$ drivers and the pickup-delivery location pair for each of all $N$ tasks, so that there is at least one driver for all $od \in \cW$ and one task for all $rs \in \cT$.
    \item Compute the private disutility $c^{rs}_a = C^{rs}_{od} - \epsilon^{rs}_a$ of driver $a \in \cA_{od}$ performing a task for $rs \in \cT$, where $C^{rs}_{od}$ is defined by \eqref{eq:detour_cost} and $\epsilon^{rs}_a$ is randomly drawn from an i.i.d. type I extreme value distribution with scale $\theta$.
\end{enumerate}
In the experiment, we set the parameters for the base setting to $(|\cW|, |\cT|, |\cA|, \theta) = (100, 100, 50000, 1.0)$ and defined $\overline{c}_{rs} = \gamma t_{rs}$ with $\gamma = 3.0$ and $t_{rs}$ being the shortest path travel time between $rs$. We performed the following three experiments:
\renewcommand{\theenumi}{\arabic{enumi}}
\renewcommand{\labelenumi}{\theenumi.}
\begin{enumerate}
    \item Change $|\cA|$ to test the scalability of the proposed approach,  
    \item Change $|\cW|$ to mainly see the effect of the number of OD pairs on computational efficiency and that of the (average) number of drivers for each OD pair on the approximation accuracy, 
    \item Change $\theta$ to see the effect of uncertainty of drivers' utility on the approach's performance. 
\end{enumerate}
For each set of parameters, we generated 30 different datasets with the above-mentioned procedure and reported the average of performance indicators.

\subsection{Performance indicators}
We used the following indicators to evaluate the performance of the proposed approach:

\begin{itemize}
    \item \textbf{Computing time}. We recorded the computing time required to solve [SO-A-P] by the Sinkhorn algorithm plus the average computing time taken to solve the sub-problems for all OD pairs. It was compared to the computing time of the naive mechanism that solved the LP problem [SO-P]. The reasoning behind taking the average over the sub-problems is that the sub-problems are completely independent of each other and can be solved in parallel.
    \item \textbf{Social surplus}. The optimal objective value $\bar{z}^\star$ of the original [SO-P] is the social surplus that we aim to accurately approximate. Therefore, we compared the sum of the objective values $z^\star = \sum_{od \in \cW} z^\star_{od}$ of our sub-problems [SO-P/Sub($od$)] to it. We also evaluated the surplus' difference for each OD pair, by comparing $z^\star_{od}$ to the OD-specific surplus $\bar{z}^\star_{od} = \sum_{a\in\cA_{od}} (\overline{c}_{rs} -c^{rs}_{a})y^{rs\star}_{a}$ of [SO-P], and reported the mean and maximum. The approximation error was measured by the true relative difference.
    \item \textbf{Task partition patterns}. We additionally compared the optimal task partition pattern $f^{rs\star}_{od}$ obtained by solving [SO-A-P] to that based on the optimal matching pattern of [SO-P], $\bar{f}^{rs}_{od} = \sum_{a\in\cA_{od}} y^{rs\star}_a$. To compare among different sizes of problems, we converted them to the partition ratios $\bar{p}^{rs}_{od} = \bar{f}^{rs}_{od} / q_{od}$ and $p^{rs}_{od} = f^{rs}_{od} / q_{od}$, and measured the Kullback-Leibler (KL) divergence:
    \begin{align}\nonumber
        D_{od}(\bar{\mathbf{p}} || \mathbf{p}) = \sum_{rs \in \cT} \bar{p}^{rs}_{od} \log \frac{\bar{p}^{rs}_{od}}{p^{rs}_{od}}.
    \end{align} 
\end{itemize}


\subsection{Results}
Figures \ref{fig:cpu_time}--\ref{fig:KL} show the results of the experiments. The reported values for each set of parameters are the average and standard deviation over 30 experiments.

\paragraph{Computing time.}
Overall, the proposed approach efficiently solved the matching problem: in most cases, it was at least 100 times faster than the naive mechanism. 
As shown in \Cref{fig:cpu_time}(a), the computing times for both the proposed approach and the naive mechanism increased with the number of drivers. Nevertheless, even with a very large case with 100,000 drivers (and 200,000 tasks), the proposed approach took on average only 1.1 seconds, which is sufficiently fast and approximately two hundred times faster than the naive mechanism whose average computing time was 210.3 seconds.

Regarding the number of OD pairs, while the computing time for the naive mechanism did not change, the computing time for the proposed approach decreased with the number of OD pairs. This is because the number of drivers for each OD pair increased with a smaller number of OD pairs, and the sub-problems of our approach required more computing time. 
This result suggests the significant efficiency of the approach in cases where the number of OD pairs is large (i.e., tasks can be partitioned with respect to many groups of drivers).

Finally, as for the scale $\theta$ of the distribution of drivers' private utility, the computing time of the naive mechanism did not change with it. The computing time for the proposed approach slightly increased as $\theta$ got larger. This is because a large $\theta$ means a small variance in drivers' private utility and the Sinkhorn algorithm for our master problem required more iterations until convergence. However, the effect was small, as the computing time for the master problem does not significantly contribute to the whole time of the approach, and indeed the difference in mean computing time between the cases with $\theta = 0.1$ and $\theta = 5.0$ was just 0.12 seconds (0.55 and 0.67 seconds, respectively).

\begin{figure}[t]
    \centering
    \includegraphics[width=0.95\textwidth]{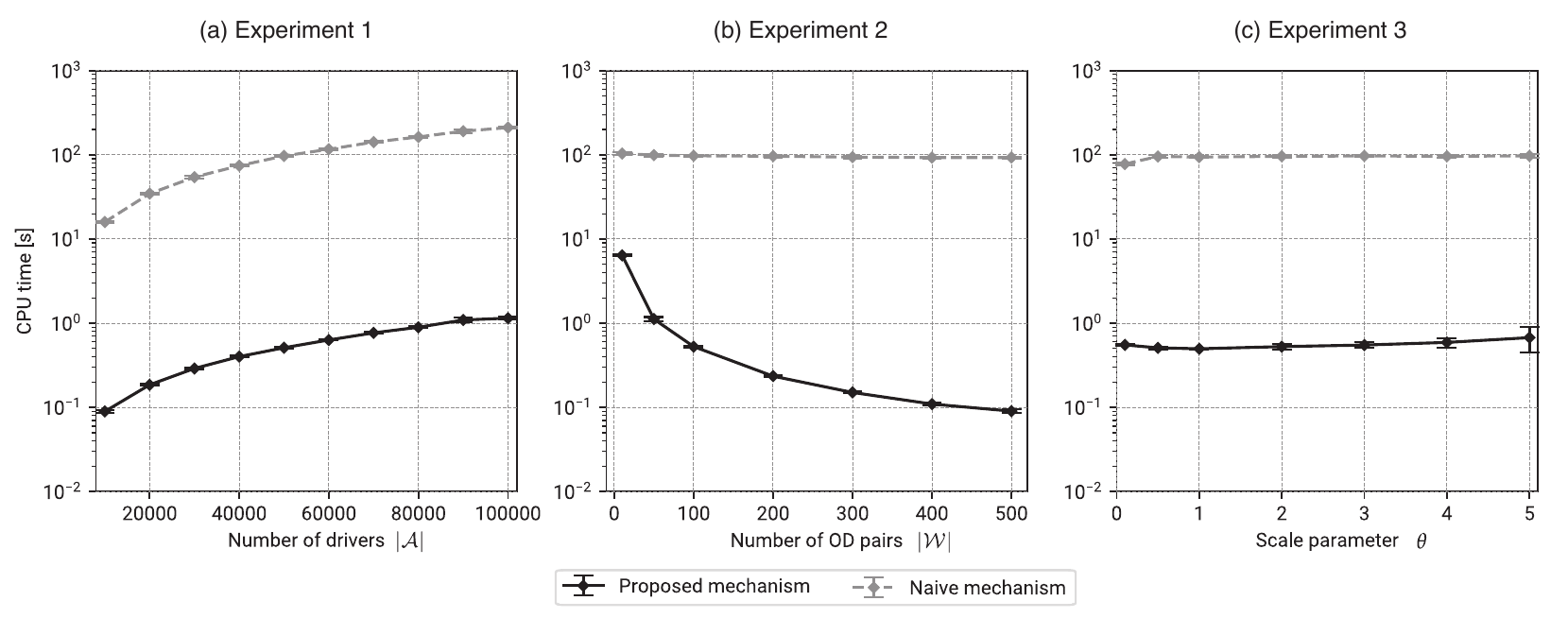}
    \caption{Computing times.}
    \label{fig:cpu_time}
\end{figure}


\begin{figure}[t]
    \centering
    \includegraphics[width=0.95\textwidth]{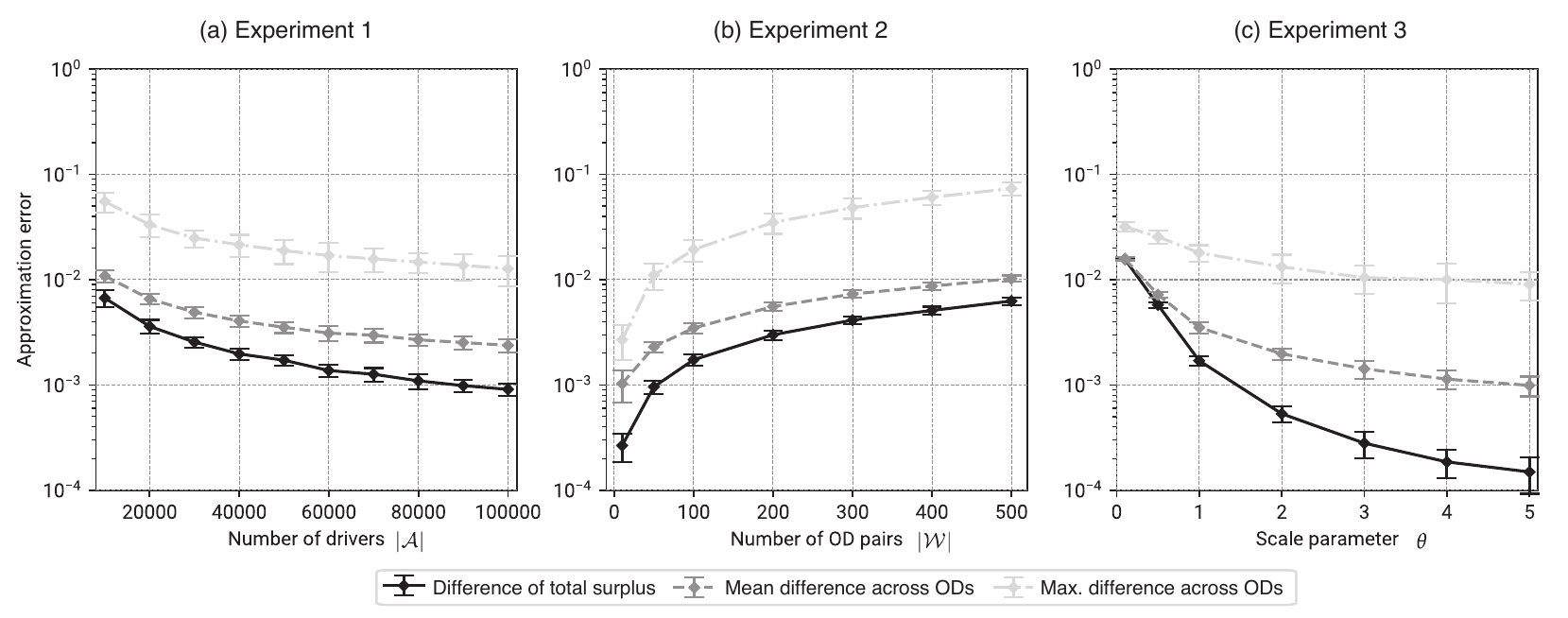}
    \caption{Approximation error: Difference in social surplus achieved.}
    \label{fig:error}
\end{figure}


\begin{figure}[t]
    \centering
    \includegraphics[width=0.95\textwidth]{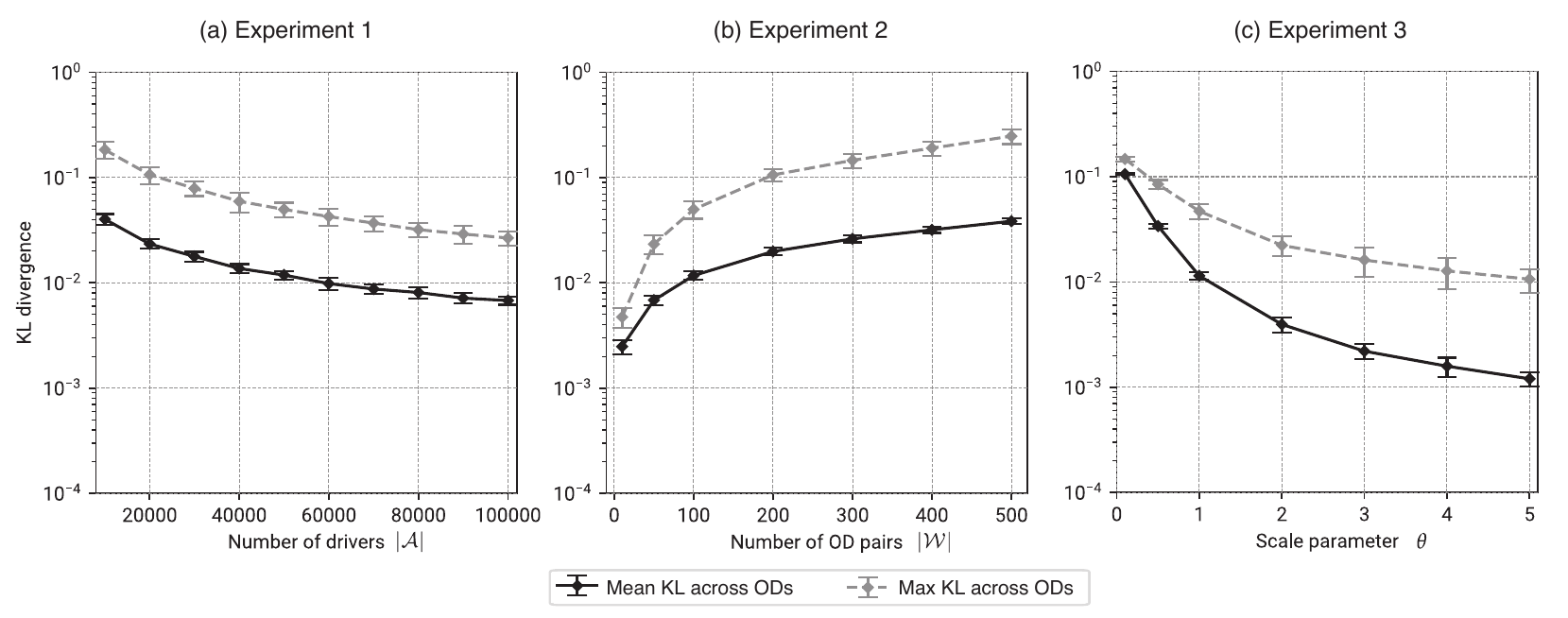}
    \caption{KL divergence: Difference in task partition patterns.}
    \label{fig:KL}
\end{figure}


\paragraph{Approximation error.}
The approximation errors of our approach mainly arise from the assumption of drivers' utility distribution and the continuous approximation of integer constraints. We focused on two errors of objective values and task partition patterns, as shown in Figures \ref{fig:error} and \ref{fig:KL}. As expected, the overall tendency observed from the results is that approximation errors got smaller when the number of drivers for each OD pair became larger. Also, the errors became large when the variance of utility distribution was large.

The objective values, i.e., the achieved social surplus were overall approximated with small errors, as exhibited by the black solid lines in \Cref{fig:error}. In the experiments with respect to the number of drivers and the number of OD pairs, the relative error between $\bar{z}^\star$ and $z^\star$ was always below 0.01, indicating an error smaller than 1\%. Only an exception was the case with $\theta = 0.1$ for the experiment with respect to $\theta$, but the error was still 1.6 \%, which is sufficiently small. When $\theta$ was equal to or greater than 2.0, the error was smaller than 0.1 \%, implying that the approximation error could be reduced if we estimated the utility distribution with a small variance by collecting the information on private utility through auctions.

In more detail, the gray and light-gray dashed lines in \Cref{fig:error} display the mean and maximum relative differences across OD-specific objective values. The mean difference was below 1\% in most cases, and the maximum difference was always below 10\%. Even when the number of OD pairs was very large ($|\cW| = 500$), the worst case still achieved the approximation error of 7.3\% (on average over 30 runs).

The error in the task partition pattern, measured by the KL divergence, showed quite similar behavior to the error in social surplus (\Cref{fig:KL}). Both the mean and maximum values of the KL divergence across the OD pairs decreased as the number of drivers, number of OD pairs, and magnitude of the scale parameter increased.

\section{Concluding remarks}\label{sec:conclusion}
This study proposed a fluid-particle decomposition approach as a novel framework for the matching market design of a CSD system. 
The approach allows us to observe drivers' private valuations for performing delivery tasks through an auction mechanism, while efficiently and accurately solving the matching problem. The numerical experiments demonstrated that our approach was at least 100 times faster than a naive LP mechanism and approximated the original objective value with errors of less than 1\%. 

As concluding remarks, we finally discuss the generalizability and limitations of the study.

While the algorithm and experiments presented in Sections \ref{sec:algorithm} and \ref{sec:experiment} relied on an MNL model, our approach is compatible with many different types of RUM frameworks as \Cref{sec:design} provided its general formulation. For instance, \cite{oyama2022markovian} presented the expected maximum utility and entropy function for the network multivariate extreme value model, and its flexible structure for capturing alternatives' correlation could better represent the distribution of drivers' utility. The approach could also be used in other applications involving large-scale matching markets such as mobility- or ride-sharing. As highlighted in \cite{alnaggar2021crowdsourced}, there is a certain similarity between en-route CSD and ride-sharing systems. Since the fluid-particle framework has not been investigated in such applications, its integration into the market design of other systems and associated methodological extensions would be an interesting future work.

As for the limitations of the study, we restricted the number of tasks that each driver performs to one. This should be relaxed for the application in more general situations of CSD, and an extended formulation and algorithm development are essential next steps. There is another strong assumption that a driver who participates in the system is always matched with a task. However, in reality, a driver may not perform any task if the reward is not sufficient compared to his/her willingness to work. The same might be true for the demand side, as shippers do not want to request delivery tasks when the shipping price, which is strongly correlated to the reward for drivers, is high. As such, the incorporation of the price elasticity of both drivers and shippers into the proposed framework would be an appealing direction.

\appendix
\section{List of notation}\label{app:notation}
Table 1 lists notations frequently used in this paper. 

\begin{table}[t]
\centering
\label{table:notation}
\caption{Notation in this paper}
\begin{tabular}{lcl}
\hline
Category & Symbol & Description  \\ \hline
Sets & $\cN$ & Set of nodes\\
& $\cL$ & Set of links \\
& $\cW$ & Set of OD pairs of drivers $\subseteq \cN\times\cN$ \\ 
& $\cT$ & Set of pickup-delivery location pairs of delivery tasks $\subseteq \cN\times\cN$ \\
& $\cA$ & Set of drivers \\
& $\cA_{od}$ & Set of drivers whose OD pair is $(o, d)$ \\ 
\hline
Parameters & $q_{od}$ & The number of drivers whose OD pair is $(o,d)$, $q_{od} = |\cA_{od}|$   \\
 & $n_{rs}$ & The number of delivery tasks whose pickup-destination pair is $(r,s)$ \\
 & $\theta$ & Scale of type I extreme value distribution \\ \hline
Variables & $y^{rs}_{a}$ & 0--1 variable that takes $1$ when driver $a$ is matched with delivery task $(r,s)$ and 0 otherwise \\
 & $f^{rs}_{od}$ & Number of delivery tasks $(r,s)$ allocated to driver group $\cA_{od}$ \\ 
 & $w_{rs}$ & Reward for drivers performing delivery tasks for $(r, s)$ \\
& $\overline{c}_{rs}$ & Operational cost of a dedicated vehicle to conduct a unit of delivery task $(r,s)$  \\
& $c^{rs}_{a}$ & Disutility for driver $a$ taking detour to perform  delivery task $(r,s)$ \\ 
& $C^{rs}_{od}$ & Deterministic detour cost of $(r, s)$ for drivers whose OD pair is $(o,d)$ \\ 
& $t_{ij}$ & Travel time of link $ij\in\cL$, or shortest path travel time between node pair $(i, j)$, $i,j\in\cN$ \\ 
\hline
\end{tabular}
\end{table}

\section{Proof of \Cref{prop:auction}}\label{app:VCG}
\begin{proof}
    We first show that the mechanism is strategy-proof, i.e., \textit{truth-telling} is a weakly-dominant strategy of each driver. 
    Let $\mathbf{y}(\mathbf{b})$ be the realized matching pattern with bid pattern $\mathbf{b}$, and recall that under the VCG mechanism described in \Cref{sec:auction}, the reward of driver $a$ with $\mathbf{b} = (\mathbf{b}_a, \mathbf{b}_{-a})$ is
    \begin{align}
        w_{a}(\mathbf{b}_a, \mathbf{b}_{-a}) = \qty(Z^\star_{-a}(\mathbf{b}) + \sum_{rs\in\cT} \overline{c}_{rs} y^{rs}_{a}(\mathbf{b})) - \max_{\mathbf{y}_{-a}} Z_{-a}(\mathbf{y}_{-a}|\mathbf{b}_{-a}).
    \end{align}
    Then, the payoff (utility) $\pi_{a}(\mathbf{b},\mathbf{b}_{-a})$ of $a\in\cA_{od}$ is
    \begin{align}
        \pi_{a}(\mathbf{b}_a,\mathbf{b}_{-a}) = w_{a}(\mathbf{b}_a,\mathbf{b}_{-a}) - \sum_{rs\in\cT} c^{rs}_{a} y^{rs}_{a}(\mathbf{b}_a,\mathbf{b}_{-a}).
    \end{align}
    Given the private valuation $\mathbf{c}_a$ of driver $a$, we want to show that the following inequality holds for an arbitrary bidding $\mathbf{b}_a, \mathbf{b}_{-a}$:
    \begin{align}
        \pi_{a}(\mathbf{c}_a,\mathbf{b}_{-a}) \ge \pi_{a}(\mathbf{b}_a,\mathbf{b}_{-a}). \label{eq:dominance_1}
    \end{align}
    This is, by definition,
    \begin{align}
        &Z^\star_{-a}(\mathbf{c}_a, \mathbf{b}_{-a}) + \sum_{rs\in\cT} \overline{c}_{rs} y^{rs}_{a}(\mathbf{c}_a, \mathbf{b}_{-a}) - \max_{\mathbf{y}_{-a}} Z_{-a}(\mathbf{y}_{-a}|\mathbf{b}_{-a}) - \sum_{rs\in\cT} c^{rs}_{a} y^{rs}_{a}(\mathbf{c}_a,\mathbf{b}_{-a}) \ge \nonumber\\
        &Z^\star_{-a}(\mathbf{b}_a, \mathbf{b}_{-a}) + \sum_{rs\in\cT} \overline{c}_{rs} y^{rs}_{a}(\mathbf{b}_a, \mathbf{b}_{-a}) - \max_{\mathbf{y}_{-a}} Z_{-a}(\mathbf{y}_{-a}|\mathbf{b}_{-a}) - \sum_{rs\in\cT} b^{rs}_{a} y^{rs}_{a}(\mathbf{b}_a,\mathbf{b}_{-a}),
    \end{align}
    and is further equivalent to
    \begin{align}
        Z^\star_{-a}(\mathbf{c}_a, \mathbf{b}_{-a}) + \sum_{rs\in\cT} (\overline{c}_{rs} - c^{rs}_{a}) y^{rs}_{a}(\mathbf{c}_a, \mathbf{b}_{-a}) \ge 
        Z^\star_{-a}(\mathbf{b}_a, \mathbf{b}_{-a}) + \sum_{rs\in\cT} (\overline{c}_{rs} - b^{rs}_{a}) y^{rs}_{a}(\mathbf{b}_a, \mathbf{b}_{-a}). \label{eq:dominance_2}
    \end{align}
    Because the optimal matching pattern $\mathbf{y}(\mathbf{c}_a, \mathbf{b}_{-a})$ holds standardness by definition, i.e.,
    \begin{align}
        \mathbf{y}(\mathbf{c}_a, \mathbf{b}_{-a}) \equiv \arg\max_{\mathbf{y}} \sum_{rs\in\cT} (\overline{c}_{rs} - c^{rs}_{a}) y^{rs}_{a} + \sum_{a' \neq a} \sum_{rs\in\cT} (\overline{c}_{rs} - b^{rs}_{a'}) y^{rs}_{a'},
    \end{align}
    then inequality \eqref{eq:dominance_2} holds. Hence, \eqref{eq:dominance_1} holds. This proves the \textit{strategy-proofness} of the mechanism.

    Once the truth-telling holds, then no driver has an incentive to make false bids. Thus, the matching realized by the VCG mechanism is the solution to  following optimization problem:
    \begin{align}
        \max_{\mathbf{y}_{(od)}}\quad& \sum_{a\in\cA_{od}}\sum_{rs \in \cT} ( \overline{c}_{rs} - c^{rs}_{a}) y^{rs}_a \label{eq:prop_matching_equiv_obj}\\
        \subto\quad&\eqref{eq:indiv_driver_consv_od},\,\eqref{eq:indiv_task_consv_od},\,\eqref{eq:y_binary_od}, \nonumber
    \end{align}
    which is equivalent to [SO-P/Sub($od$)]. This proves the matching \textit{efficiency} of the mechanism.
\end{proof}
\section{Proof of \Cref{lemma:so_d_sub}}\label{app:proof_so_d_sub}
\begin{proof}
    The constraints matrix of [SO-P/Sub($od$)] is a totally unimodular matrix because each of its elements is $0$ or $1$.
    Thus, the optimal solution of [SO-P/Sub($od$)] can be obtained by solving the LP relaxation of the problem where the binary constraints \eqref{eq:y_binary_od} are replaced with non-negative constraints, i.e., $y^{rs}_{a} \geq 0 \, (\forall a\in\cA_{od},\,\forall(r,s)\in\cT)$.
    Hence, by defining the Lagrangian as
    \begin{align}
        \label{eq:sub_opt_conditions}
        \cL(\bm{y},\bm{\sigma},\bm{\rho}) &\coloneqq
        \sum_{a\in\cA_{od}}\sum_{rs \in \cT} (\overline{c}_{rs} - c^{rs}_{a})y^{rs}_{(a,od)} 
        + \sum_{rs \in \cT} \sigma^{rs}_{od} \qty( f^{rs}_{od} - \sum_{a\in\cA}y^{rs}_{(a,od)})
        + \sum_{a\in\cA_{od}} \rho_{(a,od)} \qty(1 - \sum_{rs \in \cT} y^{rs}_{(a,od)})\\
        &=\sum_{a\in\cA_{od}}\sum_{rs \in \cT} \qty(\overline{c}_{rs} - c^{rs}_{a} - \sigma^{rs}_{od} - \rho_{(a,od)}) y^{rs}_{(a,od)}
        + \sum_{rs \in \cT} \sigma^{rs}_{od} f^{rs}_{od}
        + \sum_{a\in\cA_{od}} \rho_{(a,od)},
    \end{align}
    we can analyze the optimality conditions for [SO-P/Sub$(od)$] as follows:
    \begin{align}
        \begin{dcases}
            \pdv{\cL}{y^{rs}_{(a,od)}} = \overline{c}_{rs} - c^{rs}_{a} - \sigma^{rs}_{od} - \rho_{(a,od)} \leq 0 ,\, y^{rs}_{(a,od)} \geq 0 ,\, y^{rs}_{(a,od)} \pdv{\cL}{y^{rs}_{(a,od)}} = 0, \\
            \pdv{\cL}{\sigma^{rs}_{od}} = f^{rs}_{od} - \sum_{a\in\cA_{od}}y^{rs}_{(a,od)} = 0, \\
            \pdv{\cL}{\rho_{(a,od)}} = 1 - \sum_{rs \in \cT} y^{rs}_{(a,od)} = 0.
        \end{dcases}
    \end{align}
    The Lagrangian dual problem [SO-D/Sub$(od)$] is then obtained as
    \begin{align}
        \min_{\bm{\rho},\bm{\sigma}} \max_{\mathbf{y \ge 0}} \cL(\bm{y},\bm{\sigma},\bm{\rho}) = \min_{\bm{\rho},\bm{\sigma}} \cL(\bm{y}^\star,\bm{\sigma},\bm{\rho}),
    \end{align}
    that is, 
    \begin{align}
    [\text{SO-D/Sub$(od)$}]\qquad
    \min_{\bm{\rho},\bm{\sigma}}\quad&\sum_{a\in\cA_{od}} \rho_{(a,od)} + \sum_{rs \in \cT} f^{rs}_{od} \sigma^{rs}_{od} \label{eq:sod_d_obj}\\
    \subto\quad&\overline{c}_{rs} - c^{rs}_{a} - \sigma^{rs}_{od} \leq \rho_{(a,od)} \quad \forall a\in\cA_{od},\,\forall(r,s)\in\cT. \label{eq:sod_con}
    \end{align}
    The first term of its objective function can be expressed using \eqref{eq:sod_con} as 
    \begin{align}\label{eq:max_util_sigma}
        \sum_{a\in\cA_{od}} \rho_{(a,od)} 
        = \sum_{a\in\cA_{od}} \max_{rs \in \cT} \qty{\overline{c}_{rs} - c^{rs}_{a} - \sigma^{rs}_{od}}.
    \end{align}
    Finally, substituting \eqref{eq:max_util_sigma} into \eqref{eq:sod_d_obj} yields \Cref{lemma:so_d_sub}.
\end{proof}

While the proof was completed, it is worth discussing the interpretations of the Lagrangian multipliers $\bm{\rho},\bm{\sigma}$ by further analyzing the above optimality conditions. The first condition of \eqref{eq:sub_opt_conditions} yields
\begin{align}\label{eq:cond_sub_1}
    \begin{dcases}
        \overline{c}_{rs} - c^{rs}_{a} - \sigma^{rs\star}_{od}  = \rho^\star_{(a,od)} \quad \text{if $y^{rs\star}_{(a,od)} > 0$}\\
        \overline{c}_{rs} - c^{rs}_{a} - \sigma^{rs\star}_{od} \leq \rho^\star_{(a,od)} \quad \text{if $y^{rs\star}_{(a,od)} = 0$}
    \end{dcases} \qquad \forall a\in\cA_{od},\,\forall rs \in\cT.
\end{align}
This condition is equivalent to the following utility maximization problem of drivers when assuming $\sigma^{rs\star}_{od} = \overline{c}_{rs} - w^{rs\star}_{od}$:
\begin{align}\label{eq:cond_sub_1_}
    \begin{dcases}
        w^{rs\star}_{od} - c^{rs}_{a}  = \rho^\star_{(a,od)} \quad \text{if $y^{rs\star}_{(a,od)} > 0$}\\
        w^{rs\star}_{od} - c^{rs}_{a}  \leq \rho^\star_{(a,od)} \quad \text{if $y^{rs\star}_{(a,od)} = 0$}
    \end{dcases} \qquad \forall a\in\cA_{od},\,\forall rs \in\cT.
\end{align}
Therefore, the optimal Lagrangian multiplier $\bm{\sigma}^\star_{od}$ yields the optimal reward to be paid to drivers as:
\begin{align}
    w^{rs\star}_{od} = \overline{c}_{rs} - \sigma^{rs\star}_{od}.
\end{align}

\bibliographystyle{elsarticle-harv} 
\bibliography{cite} 

\end{document}